\documentclass[12pt,a4paper]{article}
\usepackage[english]{babel}

\usepackage{hyperref}
\usepackage{pstricks}
\usepackage{pst-node}
\usepackage{amsmath,boxedminipage}
\usepackage{amssymb}
\usepackage{graphicx}
\usepackage{enumerate}
\usepackage{color}
\usepackage[text={15cm,24cm}]{geometry}
\usepackage[normalem]{ulem}
\usepackage[ansinew]{inputenc}
\usepackage{ulem}
\usepackage{float}
\usepackage{subcaption}
 \usepackage{algorithm,algorithmic}

\usepackage{color}

\newenvironment{proof}{{\bf Proof}:\ }%
   {~\ \hfill $\Box$\vspace{0,5cm}}

    {~\ \hfill$\Box$\vspace{0,5cm}}

\newtheorem{prop}{Property}[section]
\newtheorem{theorem}{Theorem}[section]
\newtheorem{rmk}{Remark}[section]

\newtheorem{lemma}[theorem]{Lemma}

\newtheorem{coro}[theorem]{Corollary}

\newtheorem{fact}{Fact}[section]

\graphicspath{{.}{graphics/}}
\newrgbcolor{lightlightlightgray}{0.9 0.9 0.9}

\numberwithin{equation}{section}

\begin{document}
\title{The Complexity of 2-Intersection Graphs of 3-Hypergraphs Recognition for Claw-free Graphs and triangulated Claw-free Graphs}

\author{N. Di Marco\footnotemark[1], A.Frosini\footnotemark[1],
C.\ Picouleau\footnotemark[2]}
\date{\today}

\def\thefootnote{\fnsymbol{footnote}}

\footnotetext[1]{ \noindent
Dipartimento di Matematica e Informatica, Universit\`a di Firenze, Firenze, Italy,. Email: {\tt
niccolo.dimarco@unifi.it,andrea.frosini@unifi.it}}
\footnotetext[2]{ \noindent
Conservatoire National des Arts et M\'etiers, CEDRIC laboratory, Paris (France). Email: {\tt
christophe.picouleau@cnam.fr}}

\maketitle

\begin{abstract}
Given a $3$-uniform hypergraph $H$, its $2$-intersection graph $G$ has for vertex set the hyperedges of $H$ and $ee'$ is an edge of $G$ whenever $e$ and $e'$ have exactly two common vertices in $H$.
Di Marco et al. prove in  \cite{rec2inter} that deciding wether a graph $G$ is the $2$-intersection graph of a $3$-uniform hypergraph is $NP$-complete. The main problems we study concern the class of  claw-free graphs. We show that the recognition problem remains $NP$-complete when $G$ is claw-free graphs but becomes polynomial if in addition $G$ is triangulated.

 \vspace{0.2cm}
\end{abstract}

\section{Introduction}\label{intro}

A hypergraph $H=(V,E)$ \cite{Bergehyper} is a generalization of the concept of graph. In detail, it is defined considering a set of vertices $V=\{v_1,\ldots,v_n\}$ and a set $E\subset 2^{\vert V\vert}\setminus\{\emptyset\}$ of hyperedges such that $e\not \subset e'$ for any pair $e,e'$ of $E$. In the case in which  $\vert e\vert=1, \forall e \in E$ we say that $H$ is trivial. 

Similarly to the graph case, the degree of
a vertex $v\in V$ is the number of hyperedges $e\in E$
such that $v\in e$. 
Another important notion in the field of hypergraphs is that of uniformity. We say that $H$ is
$k$-uniform if $\vert e\vert=k$ for all hyperedge $e\in E$.
We also suppose that $H$ has no parallel hyperedges,
i.e., $e\ne e'$ for any pair $e,e'$ of hyperedges. Therefore, a simple
graph (loopless and without parallel edges), is a $2$-uniform
hypergraph.

In this paper we are concerned with the reconstruction of intersection graphs of hypergraphs. In particular, given a $k$-uniform hypergraph $H$, its $l$-intersection graph $G=L_k^l(H),1\le l<k$ is  $G=(E,F)$ where the vertex set is $E$, and $ee'\in F$ if and only if $\vert e\cap e'\vert=l$ (i.e.  two hyperedges of $H$ intersect by exactly $l$ elements). 
Note that a similar definition could be given for general hypergraphs. However, here we focus only on the simpler uniform case.

Starting from the previous idea, new concepts could be introduced.
For example, given a pair of vertices $u, v \in H$, their multiplicity  is the number of edges in $H$ containing both $u$ and $v$. We denote with $m(H)$ the maximum multiplicity among all pairs of vertices.
From this notion follows the concept of linear hypergraphs. In particular, $H$ is linear if and only if $m(H)=1$. 

In general, we define $L_k^l$ as the class of graphs $G$ such that there exists a $k$-uniform hypergraph $H$ such that $G=L_k^l(H)$.  In such a case we say that $H$ is a preimage of $G$ (note that a preimage is not necessarily unique). In particular, note that $L_2^1$ corresponds to the class of line graph. Finally, we denote $L_k^\infty=\bigcup_{i=1}^{k-1}L_k^i$. 

Since we are interested in finding when a graph has a preimage, we use specific labelling to label each vertex of $G$ appropriately.
In particular, for $G=L_k^l(H)$  a $\lambda_k^l$-labelling is a labelling of its vertices such that the label of a vertex $e$ is a $k$-set 

Some previous results are proved in the literature about intersection graphs. In particular in \cite{rec1inter}, Hlin\u en\'y and Kratochv\' il proved that deciding whether a graph $G$ belongs in $L_3^1$ is $NP$-complete. On the other hand, from the characterization of line graphs by Beineke \cite {Beineke} the problem of deciding whether $G$ belongs in $L_2^1$ is polynomial. 

In \cite{rec2inter} Di Marco et al. are interested in the null label problem on hypergraphs and in the reconstruction of graphs in $L_3 ^2$. In particular, in a previous paper, they prove a sufficient condition about the existence of the former \cite{nullLabel} using graphs in that family. Finally, they also proved that the problem of deciding whether a graph $G$ belongs in $L_3^2$ is $NP$-complete. 

In this work, relying on that result, we are interested in studying some subclasses of $L_3 ^2$ in which the problem could be polynomial solvable.
% First of all, we show that the hamiltonian cycle problem is $NP$-complete in the class $L_3^2$, confirming that, although a sufficient condition exists for finding the null label in hypergraphs, it cannot be easily applied in general.
In fact, one can remark that a graph $G\in L_3^2$ is $K_{1,4}$-free. A natural subclass of $K_{1,4}$-free graph is the set of claw-free graphs, thus we are interested in the characterization of them in $L_3^2$. 
In particular, here we show that deciding whether a claw-free graph belongs to $L_3^2$ remains a $NP$-complete problem, but, interestingly, it is polynomial in the subclass of triangulated graphs.

The article is organized as follows. In Section \ref{def} we give the notation and definition of graph theory we use throughout the paper, focusing also on properties of the class $L_3^2$. In section~\ref{sec:claw_free} we study the reconstruction of the subclasses of claw-free and triangulated graphs. Finally, in Section \ref{prop}, we give some properties and complexity results for the graphs in $L_k^l$.

\section{Graphs definitions and notations}\label{def}

In this section we provide some basic definitions, used throughout the paper. The reader is referred to \cite{Bondy} for definitions and notations in graph theory. 

We are concerned with simple undirected graphs $G=(V,E)$, $\vert V\vert=n,\vert E\vert=m$.  For $v\in V$, we define the open neighbourhood $N(v) = \{w \in  V \ \vert \ (v,w) \in E \}$. Similarly, we define $N[v]=N(v)\cup\{v\}$ its closed neighbourhood. The degree of $v\in V$ is $d_G(v)=\vert N(v)\vert$ or simply $d(v)$ when the context is unambiguous. We denote with $\Delta(G),\delta(G)$ the maximal, respectively minimal, degree of a vertex.
A vertex $v$ is {\it universal} if $N[v]=V$. A vertex $v$ is a {\it leaf} if $d(v)=1$. $G$ is $k$-regular when $d(v)=k$ for any $v\in V$.

For $S\subseteq V$, let $G[S]$ denote the subgraph of $G$ {\it induced} by $S$, which has vertex set~$S$ and edge set $\{uv\in E\; |\; u,v\in S\}$. For $v\in V$, we write $G-v=G[V\setminus \{v\}]$.  Similarly, for $S\subsetneq V$ and $v\in V\setminus S$ we write $G[S]+v=G[S\cup \{v\}]$. For $e\in E$, we write $G-e=(V,E\setminus \{e\})$.

For $k\geq 1$, $P_k=u_1-u_2-\cdots-u_k$ is a chordless path if no two vertices are connected by an edge that is not in $P_k$, i.e. if $V_{P_k} = \{u_1, \ldots, u_k\}$ then $G(V_{P_k}) = P_k$.
In a similar fashion, it is possible to define a chordless cycle $C_k=u_1-u_2-\cdots-u_k-u_1$ for $k \geq 3$.
For $k\ge 4$, $C_k$ is called a {\it hole}. A graph without a hole is {\it chordal} or, equivalently, {\it triangulated}.

We say that $S\subseteq V$ is called a {\it clique} if $G[S]$ is a {\it complete graph}, i.e., every pairwise distinct vertices $u,v\in S$ are adjacent. We denote with $K_p$ the clique on $p$ vertices and we say that $C_3=K_3$ is a {\it triangle}. $K_{1,p}$ is the star on $p+1$ vertices, that is, the graph with vertices $\{u,v_1,v_2\ldots,v_p\}$ and edges $uv_1,uv_2,\cdots,uv_p$. $K_{1,3}$ is a {\it claw}.

For $S\subset V$ the clique $G[S]$ is {\it maximal} if for any $v\in V\setminus S$ then $G[S]+v$ is not a clique. When the context is unambiguous a clique will be always a maximal clique. 
In this paper a clique $K$ is {\it small, medium, big} when $\vert K\vert\le 2,3\le\vert K\vert\le 4,\vert K\vert\ge5$, respectively. 
%Note that for $G$ connected and $n\ge 2$ a small clique is such that $\vert K\vert=2$ and $K=\{u,v\},uv\in E$. 

A \emph{cut-edge} in a connected graph $G$ is an edge $e\in E$ such that $G-e$ is not connected.

 For a fixed graph $H$ we write $H\subseteq G$ whenever $G$ contains an induced subgraph isomorphic to $H$. Instead, $G$ is {\it $H$-free} if $G$ has no induced subgraph isomorphic to $H$. 

 \subsection{Properties of graphs in $L_3^2$}\label{sec:claw_free}
 When dealing with reconstruction issues on graphs, it is important to consider the maximal cliques. For such reason, provided $G\in L_3^2$, we give some properties involving its maximal cliques. 

Let $K_3$, the triangle $T=K_3$ with vertices $a,b,c$. Two cases can be detected following from two different hyperedges configurations in the related $3$-uniform hypergraph: either $a=\{1,2,x\},b=\{1,2,y\},c=\{1,2,z\},x\ne y\ne z$ or $a=\{1,2,3\},b=\{1,2,4\},c=\{1,3,4\}$. The first case is defined as {\it positive clique} and the second {\it negative clique}.

A similar situation occurs with $K_4$. In fact, if we consider a clique with four vertices $a,b,c,d$, then again two cases appear: either $a=\{1,2,x\},b=\{1,2,y\},c=\{1,2,z\},d=\{1,2,t\},x\ne y\ne z\ne t$ or $a=\{1,2,3\},b=\{1,2,4\},c=\{1,3,4\},d=\{2,3,4\}$. The first case is denoted as {\it positive clique} and the second {\it negative clique}. 

For bigger cliques, the situation is easier. Let $K_p,p\ge 5,$ the clique with $p$ vertices $a_1,\ldots,a_p$. Then it must be $a_i=\{1,2,x_i\},1\le i\le p$. We defined all these cases as {\it positive}.
Note that, in general, we are referring to positive cliques whenever their labels are composed of only one sharing couple. The other cases are referred to as negative.

To set the notation, when a clique $K$ is positive we denote it by $K^+$ and $K^-$ otherwise. For convenience the clique of two vertices $K_2$ is both positive and negative.

\begin{prop}\label{posneg}
If $K_i,K_j$ are two cliques such that $\vert K_i\cap K_j\vert=2$ then we have $K_i^+$ and $K_j^-$ (or vice versa). 
\end{prop}

\begin{proof}
W.l.o.g. let us consider  $K_i\cap K_j=\{u,v\}$ with $u=\{1,2,3\},v=\{1,2,4\}$ and $s\in K_i\setminus K_j,t\in K_j\setminus K_i$ be such that $st\not\in E$. For contradiction, we assume that $K_i^+,K_j^+$ or $K_i^-,K_j^-$.
Two cases arises: if $K_i^+,K_j^+$, then it holds  $s=\{1,2,5\},t= \{1,2,6\}$ that is not possible.
Lastly, if $K_i^-,K_j^-$, then w.l.o.g., it holds $s=\{1,3,4\}$, again a contradiction when labelling  $t$.
\end{proof}

A simple consequence of the previous Property is the following.

\begin{coro}\label{twomax}

Let  $e$ be an edge of $G\in L_3^2$. Then $e$ is an edge of at most two cliques.
\end{coro}

Also the following proposition holds.

\begin{prop}\label{intcliq}
Let $K_n,K_m$ be two distinct cliques of $G=(V,E)\in L_3^2$. Then $\vert K_n\cap K_m\vert\le 2$.
\end{prop}

\begin{proof}
For contradiction, assume that $K_n\cap K_m\supseteq\{x,y,z\}$. There exists $v_1\in K_n \setminus K_n\cap K_m, v_2 \in K_m\setminus K_n\cap K_m$ such that $v_1v_2\not \in E$.

If one of the two cliques, say $K_n$, is positive then $x=\{1,2,3\},y=\{1,2,4\},z=\{1,2,5\}$. Thus $v_1=\{1,2,6\},v_2=\{1,2,7\}$, a contradiction. 

Therefore, consider $K_n=K_4^-,i=1,2$. Thus, $x=\{1,2,3\},y=\{1,2,4\},z=\{1,3,4\}$ but $v_1=\{2,3,4\},v_2=\{2,3,4\}$, another contradiction.
\end{proof}

Based on that, for two (maximal) cliques $K_i,K_j$ we say that they are {\it strongly intersecting} when $\vert K_i\cap K_j\vert=2$ and they are {\it weakly intersecting} when $\vert K_i\cap K_j\vert=1$.\\
\section{Reconstruction of claw-free graphs in $L_3^2$}\label{sec:claw_free}
 
 In this section we deal with the recognition of claw-free graphs in $L_3^2$. We initially show that there exist claw-free graphs in $L_3^2$ and then we provide some necessary conditions to check the belonging. However, we prove that, in general, the recognition problem is $NP$-complete for that class.

\subsection{Claw-free graphs in $L_3^2$}

Figure \ref{2intex} shows a claw-free graph belonging to $L_3^2$. This is not always the case, as witnessed by Figure \ref{K5-e}.
In fact, using Proposition~\ref{intcliq} it directly follows that the graph $K_5-e$ has no realization. 

\begin{figure}[H]
   \begin{center}
   \includegraphics[width=15cm, height=2.5cm, keepaspectratio=true]{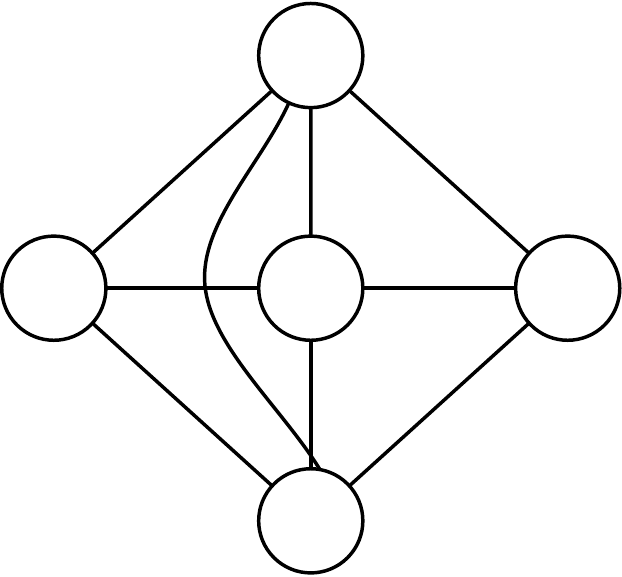}
   \end{center}
   \caption{$K_5-e$ has no realization.}
   \label{K5-e}
\end{figure}

\begin{figure}[H]
   \begin{center}
   \includegraphics[width=15cm, height=5cm, keepaspectratio=true]{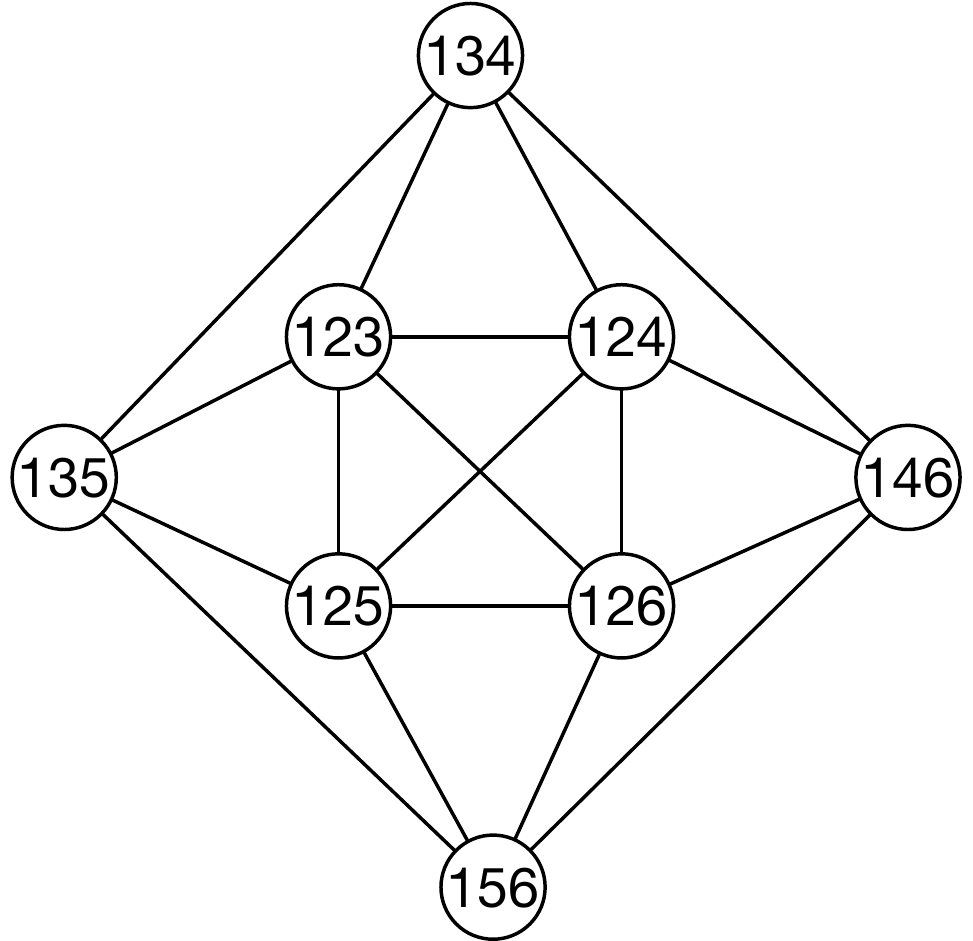}
   \end{center}
   \caption{A graph in $L_3^2$.}
   \label{2intex}
\end{figure}

We now show some examples of claw-free graphs in $L_3^2$ that will be used in the following proofs.

\begin{figure}[H]
   \begin{center}
   \includegraphics[width=12cm, height=3cm, keepaspectratio=true]{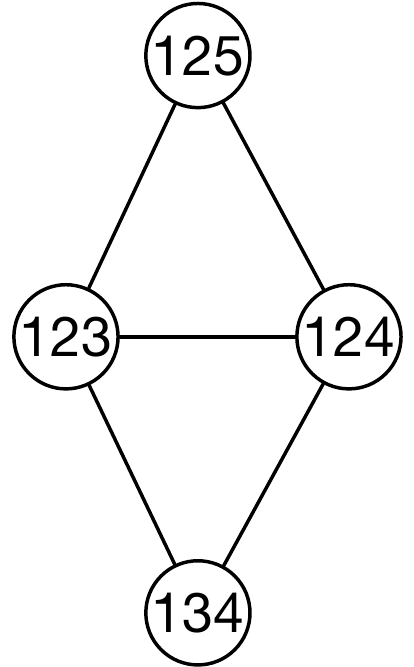}
   \end{center}
   \caption{A realization of the diamond.}
   \label{diamond}
\end{figure}

\begin{coro}\label{diam}
If $G\in L_3^2$ contains, as an induced subgraph, a diamond $D$ consisting of the two triangles $T_1,T_2$ then we have either  $T_1^-,T_2^+$ or  $T_1^+,T_2^-$.
\end{coro}

\begin{figure}[H]
   \begin{center}
   \includegraphics[width=15cm, height=2.5cm, keepaspectratio=true]{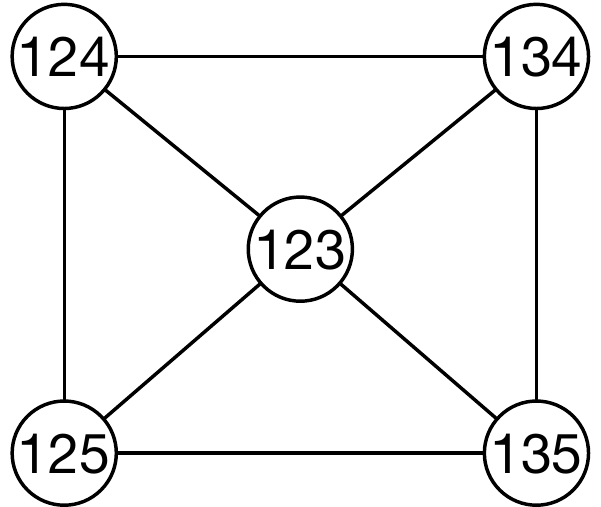}
   \end{center}
   \caption{A realization of the $4$-wheel $W_4$.}
   \label{W4}
\end{figure}

\begin{coro}\label{W4}
If $G\in L_3^2$ contains, as an induced subgraph, a
$4$-wheel $W_4$ consisting of the four triangles $T_1=\{a,b,c\},T_2=\{a,c,d\},T_3=\{a,d,e\},T_4=\{a,b,e\}$ then we have either $T_1^-,T_2^+,T_3^-,T_4^+$ or $T_1^+,T_2^-,T_3^+,T_4^-$.
\end{coro}

\begin{figure}[H]
   \begin{center}
   \includegraphics[width=15cm, height=2.5cm, keepaspectratio=true]{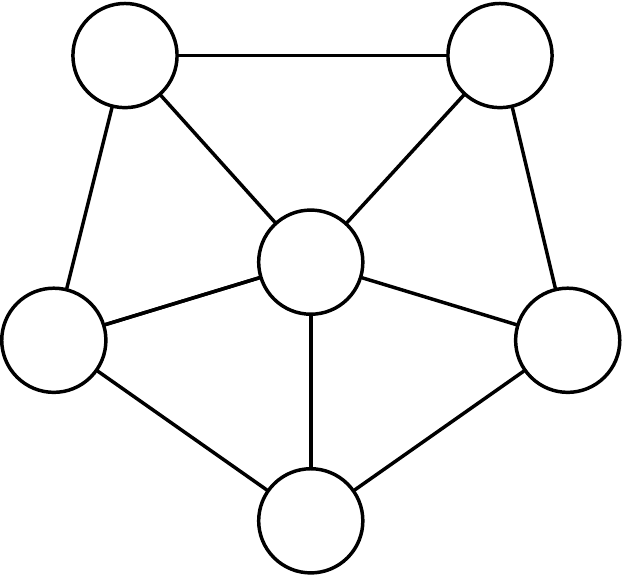}
   \end{center}
   \caption{The $5$-wheel $W_5$.}
   \label{W5}
\end{figure}

\begin{coro}\label{W5}
If $G\in L_3^2$ then it cannot contains the $5$-wheel $W_5$ as an induced subgraph.
\end{coro}

\begin{figure}[H]
   \begin{center}
   \includegraphics[width=10cm, height=2.5cm, keepaspectratio=true]{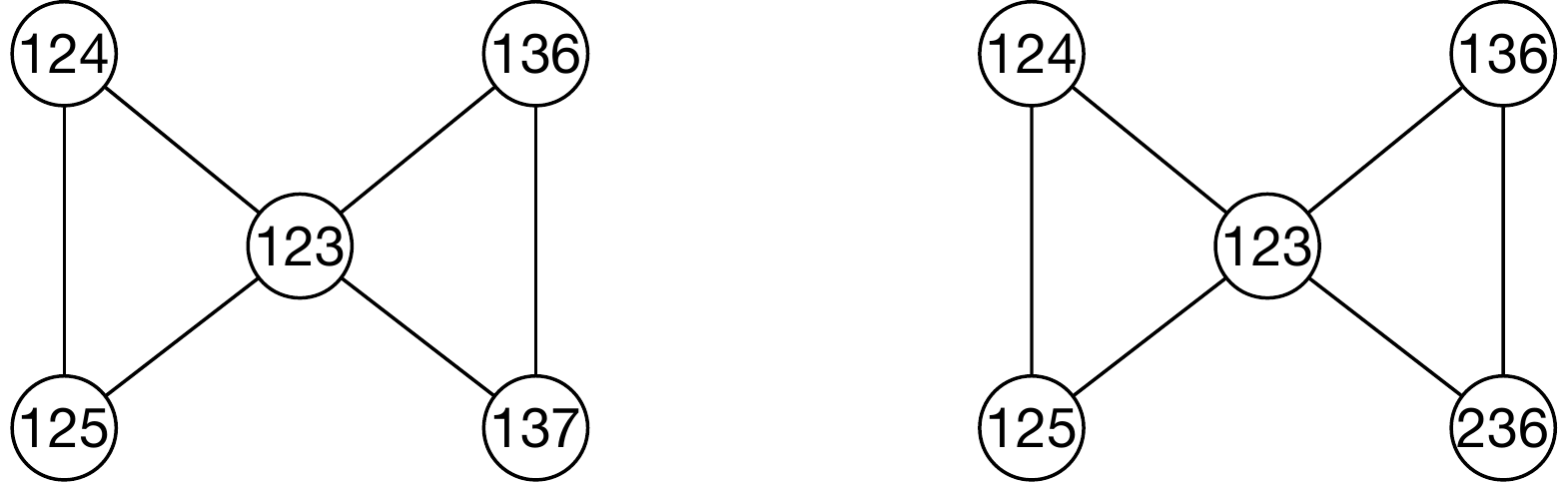}
   \end{center}
   \caption{Two realizations of the butterfly.}
   \label{butterfly}
\end{figure}

\begin{prop}\label{butter}
If $G\in L_3^2$ contains, as an induced subgraph, a butterfly $B$ consisting of the two triangles $T_1,T_2$ then we have either  $T_1^+,T_2^+$ or $T_1^-,T_2^+$ or  $T_1^+,T_2^-$. 
\end{prop}

\begin{proof}
Let $T_1=\{a,b,c\},T_2=\{a,d,e\}$. The figure \ref{butterfly} shows two realizations with $T_1^+,T_2^+$ or $T_1^-,T_2^+$ or  $T_1^+,T_2^-$. It remains to show that $T_1^-,T_2^-$ is impossible. Let $a=\{1,2,3\},b=\{1,2,4\},c=\{1,3,4\}$. We have $d=\{2,3,5\}$ but is not possible to label $T_2$ in such a way it is negative.
\end{proof}

The prism $P$ consists of two vertex disjoint triangles $T_1=\{a,b,c\}, T_2=\{d,e,f\}$ plus the three edges $ad,be,cf$.
\begin{figure}[H]
   \begin{center}
   \includegraphics[width=12cm, height=3cm, keepaspectratio=true]{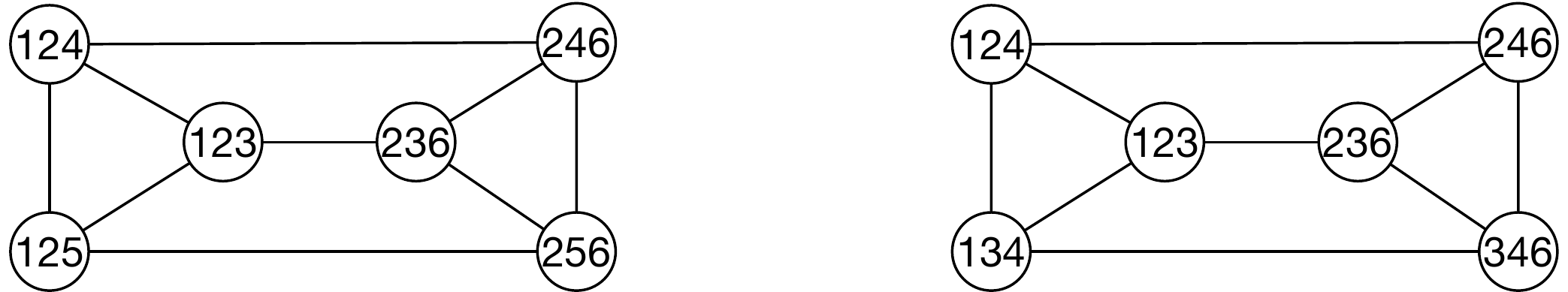}
   \end{center}
   \caption{Two realizations of the prism.}
   \label{prism}
\end{figure}

\begin{fact}\label{prism}
If $G\in L_3^2$ contains, as an induced subgraph, a prism $P$  then we have $T_1^+,T_2^+$ or $T_1^-,T_2^-$.
%If $G\in L_3^2$ contains, as an induced subgraph, a prism $P$ or  prism minus an edge $P-e$  then we have $T_1^+,T_2^+$ or $T_1^-,T_2^-$.
\end{fact}
\begin{proof}
For contradiction, assume that $T_1^+,T_2^-$. Let $a=\{1,2,3\},b=\{1,2,4\},c=\{1,2,5\}$. Without loss of generality $d=\{1,3,6\}$. Then we have $e=\{1,6,4\}$. It follows that $f=\{1,5,6\}$, so $T_2$ is positive, a contradiction.
% But $T_2$ is negative, so $e=\{1,6,7\}$ or $e=\{3,6,a\},a\ne 1$. Since $be\in E$ we obtain the contradiction.
\end{proof}

\begin{figure}[H]
   \begin{center}
   \includegraphics[width=15cm, height=3.5cm, keepaspectratio=true]{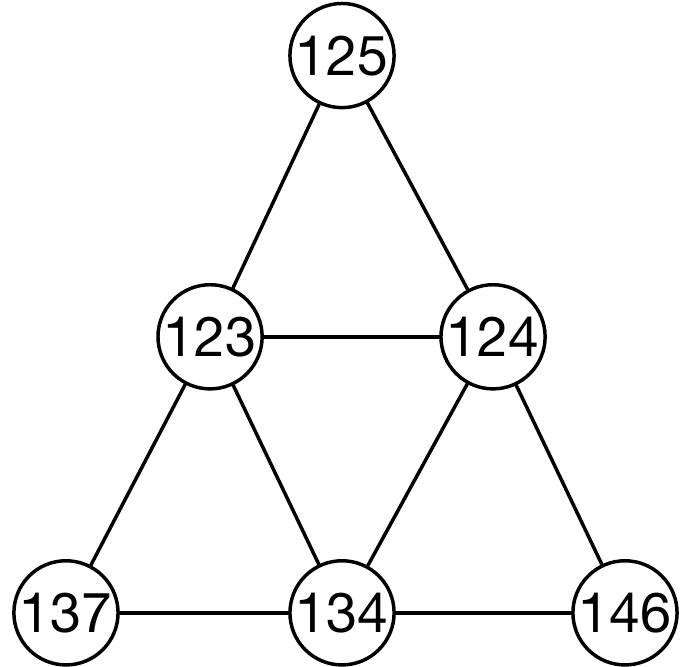}
   \end{center}
   \caption{A realization of the sun $S_3$.}
   \label{3sun}
\end{figure}

\begin{fact}\label{sun}
If $G\in L_3^2$ contains, as an induced subgraph, a
sun $S_3$ consisting of the four triangles $T_1=\{a,b,c\},T_2=\{a,b,d\},T_3=\{a,c,e\},T_4=\{b,c,f\}$ then we have $T_1^-,T_2^+,T_3^+,T_4^+$ with $a=\{1,2,3\},b=\{1,2,4\},c=\{1,3,4\},d=\{1,2,5\},e=\{1,3,6\},f=\{1,4,7\}$.
\end{fact}

\begin{proof}
From Corollary \ref{twomax} we have either $T_1^-,T_2^+,T_3^+,T_4^+$ or $T_1^+,T_2^-,T_3^-,T_4^-$.
The figure \ref{3sun} shows a realization with $T_1^-,T_2^+,T_3^+,T_4^+$. Now we assume that $T_1^+,T_2^-,T_3^-,T_4^-$ with $a=\{1,2,3\},b=\{1,2,4\},c=\{1,2,5\}$. Then, w.l.o.g.,  $d=\{1,3,4\}$. It follows that $e=\{2,4,5\}$ but $f$ cannot be labelled.
\end{proof}

\begin{fact}\label{P3}
If $G\in L_3^2$ contains, as an induced subgraph, a path on three vertices $u-v-w$ with $u=\{a,b,c\},w=\{d,e,f\}$ then $\{a,b,c\}\cap\{d,e,f\}\ne\emptyset$.
\end{fact}

As mentioned before, we denote with $K_4+v$ is the graph with five vertices $\{a,b,c,d,e\}$ where $\{a,b,c,d\}$ is complete and $e$ is connected to exactly one vertex, say $a$.

\begin{fact}\label{K4+v}
If $G\in L_3^2$ contains $K_4+v$, as an induced subgraph, then the clique $K_4$ of $K_4+v$ is positive.
\end{fact}
\begin{proof}
For contradiction, assume that the clique on four vertices $G[\{a,b,c,d\}]$ is negative: $a=\{1,2,3 \},b=\{1,2,4\},c=\{ 1,3,4 \},d= \{ 2,3,4 \}$. Then, without loss of generality, $e=\{1,2,5\}$. So $be\in E$, a contradiction.
\end{proof}

Note that all the graphs we have considered in this subsection are claw-free.

\subsection{Recognition for claw-free graphs in $L_3^2$}\label{sec:clawfree}
In \cite{rec2inter} authors prove that recognizing whether a graph $G$ is in $L_3^2$ is $NP$-complete. Since a graph in $L_3^2$ is $K_{1,4}$-free we are concerned with the subclass of claw-free graphs ($K_{1,3}$-free graphs).

We will prove that the problem of deciding if $G\in L_3^2$ is $NP$-complete. To reach our goal we need an intermediate problem that is defined and proved $NP$-complete below.

\subsubsection*{The $2$-labelling intersection ($2LI$) problem}

Let us consider a simple graph $G=(V,E)$ and a partition of its edge-set  $E$ into two subsets $E_w$ and $E_s$, i.e. $E=E_w\cup E_s$
and $E_w\cap E_s=\emptyset$. We call {\em weak edges}
the edges in $E_w$ and {\em strong edges} those in $E_s$.

We define a function $\varphi$, say a {\em $2$-labelling}, that
associates to each vertex $v\in V$ a pair of labels
$\{a_v,b_v\}$ such that:

\begin{description}
\item{$i)$} if $v_1\ne v_2$, then $\varphi(v_1)\ne
\varphi(v_2)$;
\item{$ii)$} if $v_1v_2\in E_w$, then $\varphi(v_1) \cap
\varphi(v_2) =\emptyset$;
\item{$iii)$} if $v_1v_2\in E_s$, then $|\varphi(v_1) \cap
\varphi(v_2)|=1$.
\end{description}

See Fig.\ref{2lab} for an example.

\begin{figure}[H]
   \begin{center}
   \includegraphics[width=15cm, height=3.5cm, keepaspectratio=true]{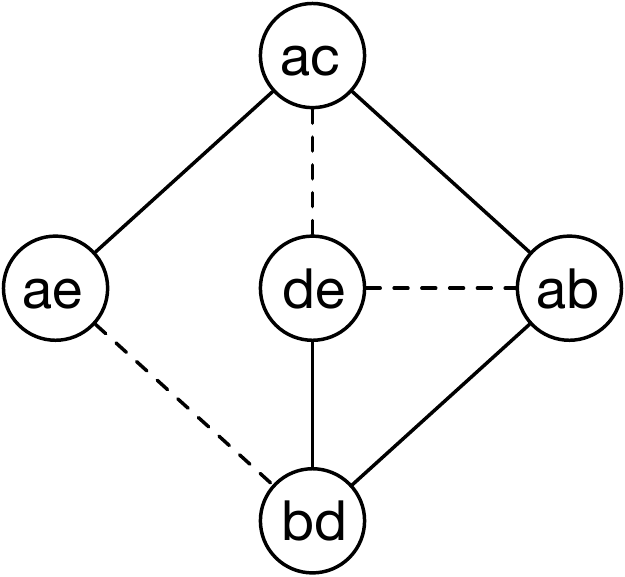}
   \end{center}
   \caption{A graph with a $2$-labelling
$\varphi$. Weak and strong edges are
represented by dotted and straight lines, respectively, while the
nodes show the pair of elements associated by $\varphi$.}
   \label{2lab}
\end{figure}

The $2LI$ problem,
provided in its decision form, follows

\begin{center}
\begin{boxedminipage}{.99\textwidth}
\textsc{\sc  $2$-labelling intersection ($2LI$)} \\[2pt]
\begin{tabular}{ r p{0.8\textwidth}}
\textit{~~~~Instance:} &a simple graph $G=(V,E)$ and a partition of its
edges into $E_w$ and $E_s$.\\
\textit{Question:} &does $G$ admit $\varphi$ a $2$-labelling of its vertices?
\end{tabular}
\end{boxedminipage}
\end{center}

We show the NP-completeness of $2LI$ by a reduction that involves the problem $3-SAT$ (LO2 in \cite{GJ})

\begin{center}
\begin{boxedminipage}{.99\textwidth}
\textsc{\sc  $3$-SAT} \\[2pt]
\begin{tabular}{ r p{0.8\textwidth}}
\textit{~~~~Instance:} &a set $U$ of variables, a collection $C$ of
clauses over $U$ such that each clause $c\in C$ has $|c|=3$.\\
\textit{Question:} &Is there a satisfying truth assignment for $C$?
\end{tabular}
\end{boxedminipage}
\end{center}

Given an instance $A$ of $3$-Sat, we construct a graph $G_A=(V_A,E_A)$
and a partition of its edges into weak and strong edges
$E_{A,w}$ and $E_{A,s}$ such that the $3$-Sat instance admits a
solution if and only if $G_A$ admits a $2$-labelling. This will
imply the NP-completeness of $2LI$.

Hence, we start by providing two graphs' prototypes to model variables
and clauses of the $3$-Sat instance $A$, then we show how to use them
to reach the graph $G_A$.

\medskip
\noindent {\em Representing the variables and the clauses of $A$}
\medskip

Let us define the graph $G_x=(V_x,E_x)$ that will be used to represent
each variable $x\in U$. The set $V_x$ consists of three vertices
$v_1^x$, $v_2^x$, and $v_3^x$, while $E_x$ is partitioned into the
weak edges $E_{x,w}=\{v_2^xv_3^x\}$ and the strong
edges $E_{x,s}=\{v_1^xv_2^x,v_1^xv_2^x\}$ (see
Fig.\ref{gadgetvar}, $(a)$).

\begin{figure}[H]
   \begin{center}
   \includegraphics[width=15cm, height=5cm, keepaspectratio=true]{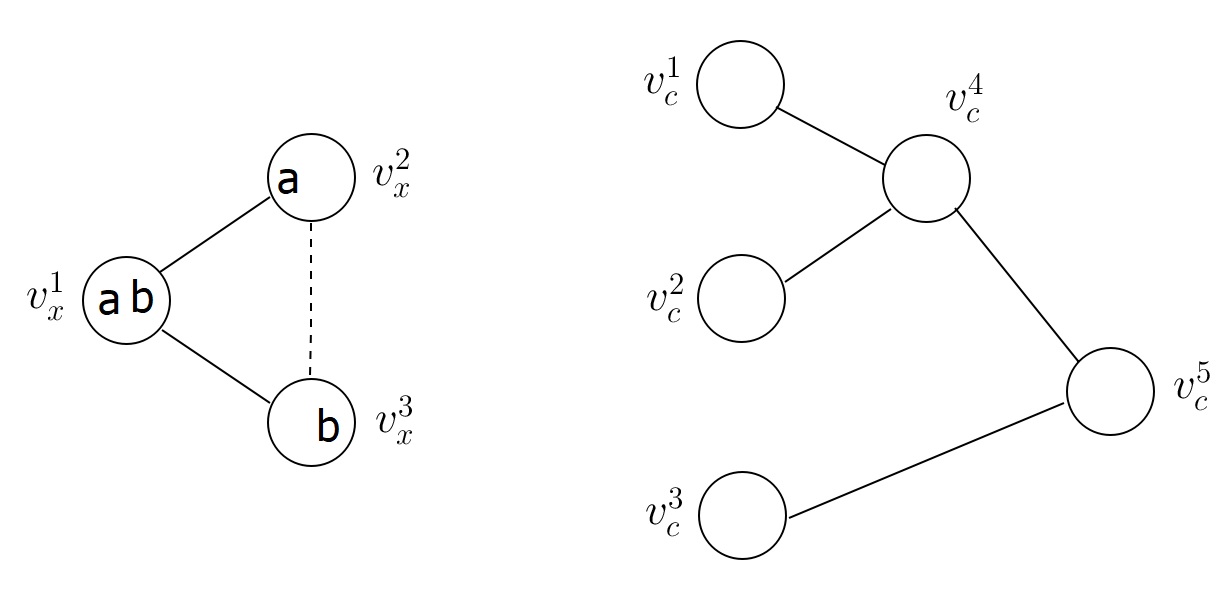}
   \end{center}
   \caption{$(a)$ the graph defined for a variable $x\in U$; (b) the
graph defined for a clause $c\in C$.  In $(a)$ the
nodes show the pair of elements associated by $\varphi$.}
   \label{gadgetvar}
\end{figure}

\begin{lemma}\label{lem:element}
Let $G_x$ be the graph related to variable $x$ and
$\varphi(v_1^x)=\{a,b\}$. It holds, w.l.o.g., that
$a\in \varphi(v_2^x)$ and $b\in \varphi(v_3^x)$.
\end{lemma}

The proof is immediate by definition of weak and strong edges.\\

The gadget $G_x$ will also be used to represent the dichotomy of the
truth values in the final graph. In particular, from now on, we
consider the truth values labels $T$ and $F$. We define
$G_{TF}$ similarly to $G_x$ with the further assumption that $a=T$ and
$b=F$.

%\medskip
%\noindent {\em Representing the clauses $C$ of $A$}
%\medskip

On the other hand, we associate to each clause $c\in C$ a graph
$G_c=(V_c,E_c)$ having five vertices
$V_c=\{v_c^1,v_c^2,v_c^3,v_c^4,v_c^5\}$ and four strong connections, i.e., $E_c=E_{c,s}=\{v_c^1v_c^4,v_c^2v_c^4,v_c^3v_c^5,v_c^4v_c^5\}$
(see Fig.\ref{gadgetvar}, $(b)$).

\medskip
\noindent {\em Putting things together}
\medskip

So, starting from an instance $A$ of $3$-Sat, we proceed in defining
the graph $G_A=(V_A,E_A)$. The reader can follow the construction in
Fig.\ref{finalgraph}. We include in $G_A$ $n=|U|$ graphs
$G_{x_1},\dots,G_{x_n}$ representing the variables of $U$, $m=|C|$
graphs  $G_{c_1},\dots,G_{c_m}$ representing the clauses of $C$, and
the graph $G_{TF}$.

The connections between these graphs in $G_A$ are set according to the
following rules:
\begin{description}
\item{$(1)$} the variables are connected by all the possible weak
edges between the vertices $v_x^1$, i.e., for each couple of variable
$x$ and $y$ in $U$, we set the weak edge $v_x^1v_y^1\in E_{A,w}$;
\item{$(2)$} let $x$ be the $i$-th variable involved in the clause
$c$, with $1\leq i \leq 3$. We set the weak edge $v_x^1v_c^i\in
E_{A,w}$, and the strong edge $v_x^3v_c^i\in E_{A,s}$ if $x$ is
negated, $v_x^2v_c^i\in E_{A,s}$ otherwise;
\item{$(3)$} for each variable $x$, we set the strong edges
$v_x^2v^1_{TF},v^3_xv^1_{TF}\in E_{A,s}$ to connect the variable
to the truth values in $G_{TF}$. Furthermore, we set three more strong
edges $v^4_cv^1_{TF},v^5_cv^1_{TF},v^5_cv^2_{TF}\in E_{A,s}$
to connect also each clause $c$ to the truth values in $G_{TF}$.
\end{description}

\begin{figure}[H]
   \begin{center}
   \includegraphics[width=15cm, height=8cm, keepaspectratio=true]{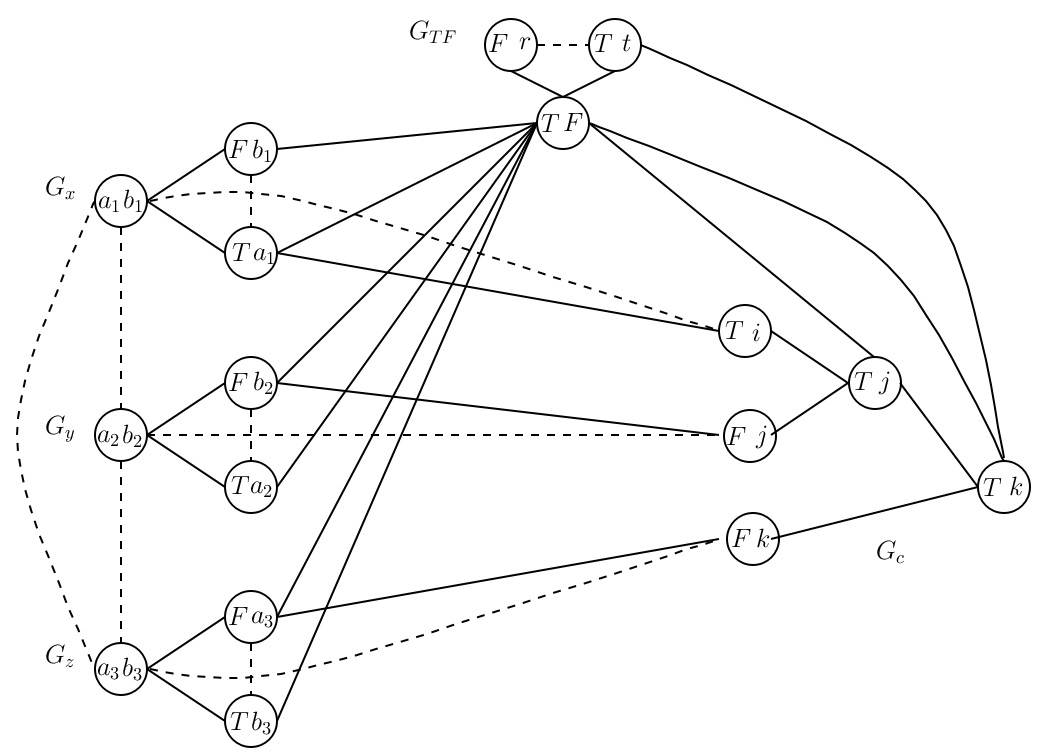}
   \end{center}
   \caption{The graph $G_A$ for a clause $C={\bar x}\vee y\vee z$}
   \label{finalgraph}
\end{figure}

\begin{theorem}\label{2label}
Given an instance $A$ of $3$-Sat, the graph $G_A$ admits a $2$-labelling
if and only if the instance $A$ has a solution.
\end{theorem}

\begin{proof}
Suppose that a $2$-labelling $\varphi$ exists. Let $\{a_i,b_i\}$ be
the label associated to the node $v^1_{x_i}$ of $G_{x_i}$. The
connections in $(1)$ assure that the labels $\{a_i,b_i\},1\le i\le n,$ have no common elements.

By the edges in $(2)$, for each variable $x$, one among $v^1_x$ and
$v^2_x$ contains $T$, while $F$ belongs to the other since their
labels can not  intersect by definition of $G_x$. Let us consider a
clause $c$ involving literals of the (distinct) variables $x$, $y$ and
$z$. $\varphi(v_c^1)$ contains, by the edges in $(2)$, the truth
value either in $\varphi(v_x^2)$ or in $\varphi(v_x^3)$ of $G_x$ to
whom it is strongly connected. Note that $\varphi(v_c^1)$ and
$\varphi(v_x^1)$ does not intersect since a weak edge is set
between them in $(2)$.

Now, consider the node $v_c^4$: it is strongly connected both with
$v^1_{TF}$, $v_c^1$, and $v_c^2$, so its label contains $T$ or $F$
according to the values of $\varphi(v_c^1)$ and $\varphi(v_c^2)$. More
precisely, the following three cases arise:
\begin{description}
\item{-} $T\in \varphi(v_c^1)$ and $T\in \varphi(v_c^2)$: it follows
that $T\in \varphi(v_c^4)$;
\item{-} $T\in \varphi(v_c^1)$ and $F\in \varphi(v_c^2)$, or
conversely: it follows that one among $T$ or $F$, but not both,
belongs to $\varphi(v_c^4)$;
\item{-} $F\in \varphi(v_c^1)$ and $F\in \varphi(v_c^2)$: it follows
that $F\in \varphi(v_c^4)$.
\end{description}

Finally, consider the node $v_c^5$: since it is strongly connected to
$v^1_{TF}$ and $v^2_{TF}$, its label contains, w.l.o.g. $T$. By the
three cases above, if $T\in \varphi(v_c^5)$, then it holds $T\in
\varphi(v_c^3)$ or  $T\in \varphi(v_c^4)$.
So, $F\in\varphi(v_c^1)$, $F\in \varphi(v_c^2)$, and $F\in
\varphi(v_c^3)$ if and only if $G_c$ does not admit a
$2$-labelling, and the same holds for  $G_A$. Since the three truth
values in $\varphi(v_c^1)$, $\varphi(v_c^2)$, and $\varphi(v_c^3)$ are
the truth values of the literals in $c$, then a truth assignment for
$c$ exists if and only if a $2$-labelling for $G_c$ does.

It is clear that the converse holds, as the construction is reversible.
\end{proof}

\begin{theorem}
Let $G=(V,E)$ be a claw-free graph. Deciding whether there exists a $3$-uniform hypergraph $H$ such that $G=L_3^2(H)$ is a $NP$-complete problem.
\end{theorem}
\begin{proof}
The proof has two parts. Firstly, we define $\cal C$ a subclass of claw-free graphs we are interested here. Then we show that the problem of deciding wether $G\in \cal C$ is such that $G\in L_3^2$ is equivalent to the problem $2LI$ which is $NP$-complete from Theorem  \ref{2label}.

First: the definition of $\cal C$. A graph $G\in \cal C$ consists of components $C_1,\ldots,C_k$ each of the $C_i$'s being a clique of size at least five, the $C_i$'s form a partition of the vertex set of $G$. When two components $C_i,C_j$ are linked there are connected by either a {\it strong link} or a {\it weak link}. A strong link consists of a $C_4$ of $G$ with its two non adjacent vertices $i_1,i_2\in C_i$ and  its two other non adjacent vertices $j_1,j_2\in C_j$. A weak link consist of a $K_4$ of $G$ with two vertices $i_1,i_2\in C_i$ and the two other vertices $j_1,j_2\in C_j$. The links, weak or strong, have no common vertices. It follows that $G$ is claw-free. Moreover, since $\vert C_i\vert \ge 5$, $C_i\cup C_j$ the union of two distinct components cannot be a clique.

When $G\in L_3^2$ the components satisfy the following: Since each component $C_i=\{v_1^i,v_2^i,\ldots,v_p^i\}$ has at least five vertices, we necessarily have $v_1^i=\{i,i',1\},v_2^i=\{i,i',2\},\ldots,v_p^i=\{i,i',p\}$ and $C_i$ is associated with its pair of common labels $\{i,i'\}$. For two distinct components $C_i,C_j$ we have $\vert \{i,i'\}\cap \{j,j'\}\vert\le1$. When $C_i,C_j$ are connected with a strong link then we have $\vert \{i,i'\}\cap \{j,j'\}\vert=1$. When $C_i,C_j$ are connected with a weak link then we have $\{i,i'\}\cap \{j,j'\}=\emptyset$.

Second: equivalence with the problem $2LI$. Given $G\in \cal C$ we define the graph $G'$ as follows: to the vertices $v_i$ of $G'$ correspond the components $C_i$ of $G$, and vice versa; to a strong (resp. weak) link of $G$ corresponds a strong (resp. weak) edge of $G'$, and vice versa.

We assume that there exists a $3$-uniform hypergraph $H$ such that $G=L_3^2(H)$. Since $\vert C_i\vert \ge 5$ the intersection of the labels of the pairs of vertices in the same component $C_i$ is the same two labels says $\{i,i'\}$. Now, for two distinct $C_i,C_j$, since $C_i\cup C_j$ is not a clique we have that $\vert \{i,i'\}\cap\{j,j'\}\vert\le 1$. When $C_i,C_j$ are strongly connected then $\vert \{i,i'\}\cap\{j,j'\}\vert= 1$, when they are weakly connected then $\vert \{i,i'\}\cap\{j,j'\}\vert= 0$. Thus for each vertex $v_i$ of $G'$ when assigning the two labels $i,i$ to $v_i$ we obtain a positive answer for the problem $2LI$.

Now, we assume that the problem $2LI$ has a positive answer. Let $i,i'$ be the two labels assigned to $v_i$ in $G'$. We assign $\{i,i'\}$ to the component $C_i$ of $G$. Let $v_i,v_j$ be two vertices linked with a strong edge. Then the labels associated to $C_i, C_j$ are $\{i,i'\},\{i,j\}$, respectively. Let $w_1^i, w_2^i$ and $w_1^j, w_2^j$ be respectively the two vertices of the strong link between $C_i$ and $C_j$. Then $w_1^i=\{i,i',a\},w_2^i=\{i,i',b\},w_1^j=\{i,j,a\},w_2^j=\{i,j,b\}$. Let $v_i,v_j$ be two vertices linked with a weak edge. Then the labels associated to $C_i, C_j$ are $\{i,i'\},\{j,j'\}$, respectively.  Let $w_1^i, w_2^i$ and $w_1^j, w_2^j$ be respectively the two vertices of the weak link between $C_i$ and $C_j$. Then $w_1^i=\{i,i',j\},w_2^i=\{i,i',j'\},w_1^j=\{j,j',i\},w_2^j=\{j,j',i'\}$. Thereafter, when a vertex $w_k^i\in C_i$ is not contained in a link, weak or strong, we take $w_k^i=\{i,i',k\}$. Thus there exists $H$ a $3$-uniform hypergraph $H$ such that $G=L_3^2(H)$.
\end{proof}

\subsection{Recognition for triangulated claw-free graphs in $L_3^2$}

Recall that a graph $G$ is {\it triangulated} (or {\it chordal}) if it is $C_k$-free, $k\ge 4$.

It is known that to each triangulated graph $G=(V,E)$, we can associate, in linear time \cite{maxcliq}, a maximal clique tree $T=(C,S)$ where each maximal clique of $G$ corresponds to a vertex $c\in C$ and $cc'\in S $ if $c\cap c'\ne \emptyset$ and $c\cap c'$ is a minimal separator of $G$ which is a clique.
An example is shown in Figure~\ref{fig:clique_tree}

\begin{figure}[!ht]
    \centering
    \includegraphics[width = \linewidth]{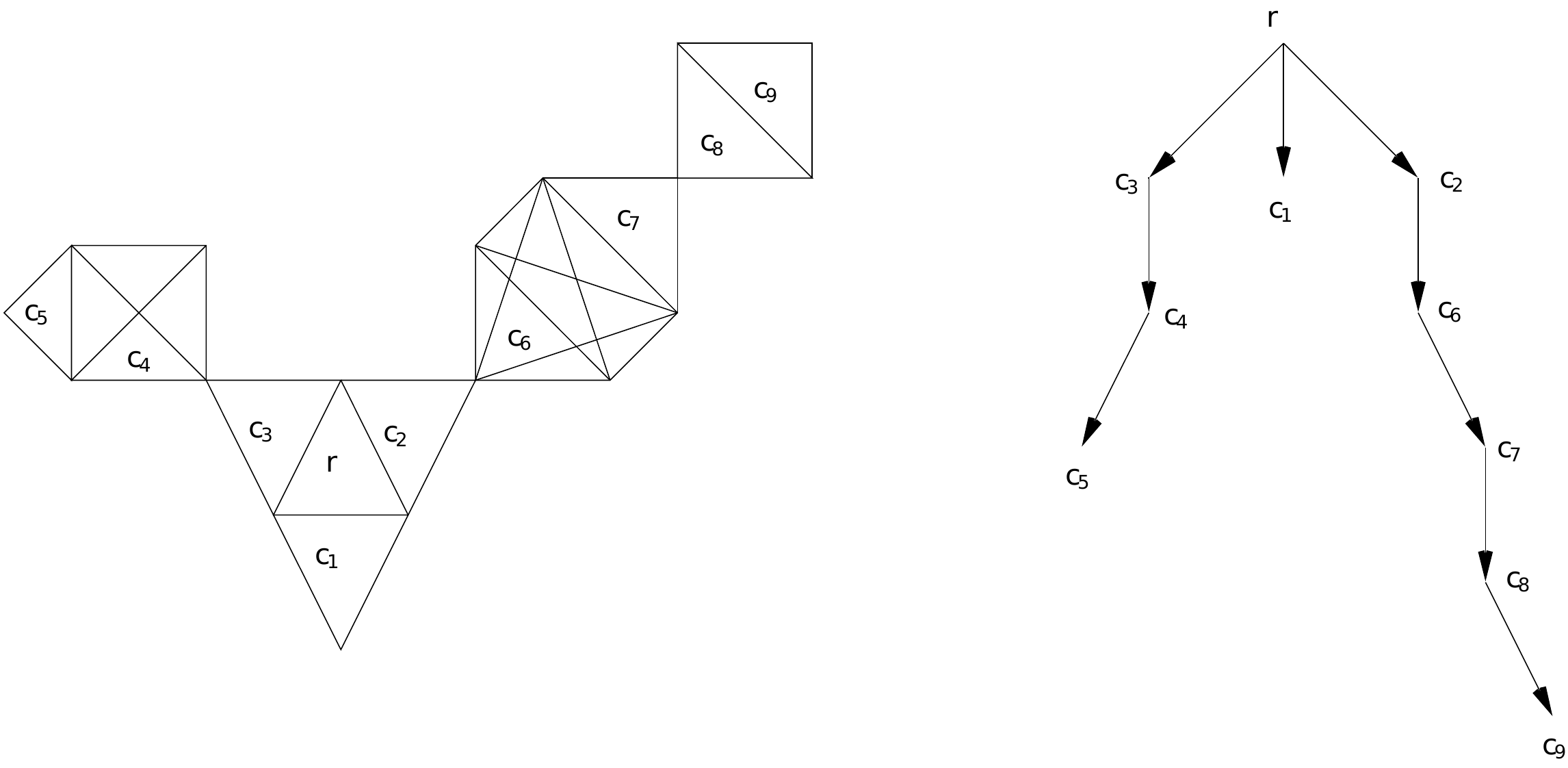}
    \caption{Example of a claw-free triangulated graph and its associated directed tree. We consider the clique $r$ as the root of the tree. The three leaves are highlighted.}
    \label{fig:clique_tree}
\end{figure}

From Property \ref{intcliq} if $G\in L_3^2$, then an edge $cc'\in S$ corresponds either to a strong intersection when $\vert c\cap c'\vert =2$ or to a weak intersection when $\vert c\cap c'\vert =1$. 
Moreover, when $cc'\in S,\vert c\vert,\vert c'\vert\ge 3$ we can easily obtain the following:

\begin{itemize}
\item  $c$ and $c'$ weakly intersect: let $c\cap c'=\{a\}$; there exists $\{u,v\}\in c, \{u',v'\}\in c'$ such that $G[\{a,u,v,u',v'\}]$ is a butterfly, otherwise $a$ cannot be a separator of $G$;
\item  $c$ and $c'$ strongly intersect: let $c\cap c'=\{a,b\}$; there exists $u\in c, u',\in c'$ such that $G[\{a,b,u,v\}]$ is a diamond, otherwise $\{a,b\}$ cannot be a separator of $G$.
\end{itemize}

In the sequel, for the sake of clarity, we denote the cliques as $K_i$ but, when not indicated, they do not represent cliques of size $i$. 
The following fact holds.

\begin{fact}\label{intclique}
Let $G=(V,E)$ be a triangulated claw-free graph and $K_t$ be a clique of $G$ such that $t\ge 4$. If $K_i,K_j$ are two (distinct) cliques that strongly intersect $K_t$ then $K_i\cap K_j\cap K_t=\emptyset$.
\end{fact}

\begin{proof}
For contradiction, we assume $v\in K_i\cap K_j\cap K_t$. Let $K_i\cap K_t=\{v,v_i\},K_j\cap K_t=\{v,v_j\},v_i\ne v_j$. Since $t\ge 4$ there exists $k\in K_t\setminus \{v,v_i,v_j\}$.  Let $w_i\in K_i\setminus \{K_j\cup K_t$\} and $w_j\in K_j \setminus \{K_i\cup K_t\}$. We have $w_iw_j\in E$, otherwise $G[\{v,k,w_i,w_j\}]$ is a claw. We also have $v_i v_j \in E$ for a similar reason. But then  $G[\{v_i,w_i,w_j,v_j\}]=C_4$ which is not possible since $G$ is triangulated.
\end{proof}

Figure \ref{Fact56} shows the situation described  in Fact \ref{intclique}.

The following theorem states the polynomiality of the reconstruction of a $3$-uniform hypergraph from a triangulated claw-free graphs $G$ such that $G=L_3^2(H)$. To help the reader, the proof is divided into small steps that lead to the final result.

\begin{figure}[H]
   \begin{center}
   \includegraphics[width=12cm, height=6cm, keepaspectratio=true]{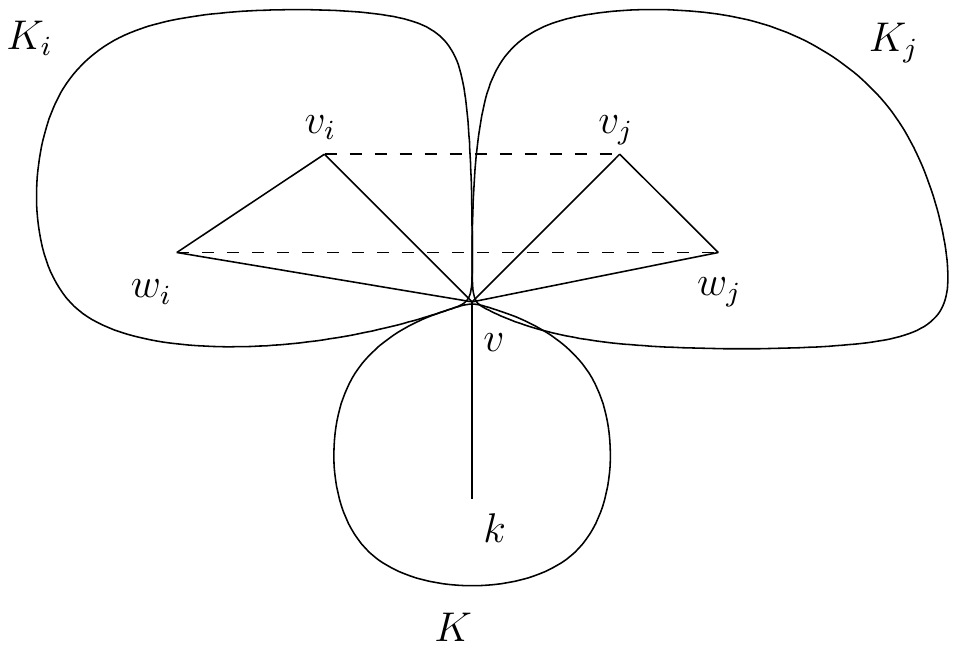}
   \end{center}
   \caption{Visual example of the proof of Fact \ref{intclique}. Curves represents nodes inside the same clique. Dashed edges are forced by claw-free property.}
   \label{Fact56}
\end{figure}

\begin{theorem}\label{triang}
Let $G=(V,E)$ be a triangulated claw-free graph. Deciding whether a $3$-uniform hypergraph $H$ exists such that $G=L_3^2(H)$ and constructing $H$ when it exists can be done in  polynomial time.
\end{theorem}

\begin{proof}
Let $G=(V,E)$ be a triangulated claw-free graph. We can assume that $G$ is connected and $\vert V\vert\ge 3$. We first consider two cases.

\vspace{1mm}
{\it G has an edge cut} 
\vspace{1mm}

Let $e=v_1v_2\in E$ be an edge-cut of $G$. We call ${\widetilde G_1},{\widetilde G_2}$ the two components of ${\widetilde G}=G-e$, with $v_1 \in {\widetilde G_1}, v_2 \in {\widetilde G_2}$. 

We will denote $G_1 = {\widetilde G_1}+v_2$ and $G_2 = {\widetilde G_2}+v_1$. 

Since $G$ is claw-free, $v_1$ is a vertex of at most two cliques: the clique $G[\{v_1,v_2\}]$ and one clique $K_n ^1$ of ${\widetilde G_1}$. The same holds for $v_2$, which can belong respectively to one clique $K_m ^2$ of ${\widetilde G_2}$ or to $G[\{v_1,v_2\}]$.  

First, we suppose that $v_1$ is a leaf of $G$ (i.e. $deg_G (v_1) = 1$). In such a case, obviously $G=G_2$. 

\begin{fact}
$G$ has a realization if and only if ${\widetilde G_2}$ has a realization such that $K_m ^2$ is positive when $m =4$.
\end{fact}

\begin{proof}
 From Fact \ref{K4+v} if  $G$ has a realization then $K_m ^2$ is positive when $m =4$. 
 
 Conversely, assume that ${\widetilde G_2}$ such that $K_m ^2$ is positive when $ m =4$. The following cases arise:
 
 \begin{itemize}
 
     \item if $m=4$ suppose that $K_m ^2=\{v_2,v_3,v_4,v_5\}$ with $v_2=\{1,2,3\},v_3=\{1,2,4\},v_4=\{1,2,5\},v_5=\{1,2,6\}$. Then we can take $v_1=\{1,3,0\}$;
     
     \item if $m \ge 5$ then $K_m ^2$ is positive and we do as before.
     
     \item if $m=3$ we have two further cases. 
     
     If $K_m ^2$ is positive, we do as above. If instead it is negative, let $K_m ^2=\{v_2,v_3,v_4\}$ with $v_2=\{1,2,3\},v_3=\{1,2,4\},v_4=\{1,3,4\}$. Then we take $v_1=\{2,3,0\}$;
     
     \item Finally $K_m ^2=\{v_2,v_3\}$ with $v_2=\{1,2,3\},v_3=\{1,2,4\}$. We take $v_1=\{1,3,0\}$.
     
 \end{itemize}
 
\end{proof}

We now suppose both $v_1$ and $v_2$ are not leaves. 

%Then both $G_1$ and $G_2$ are non trivial subgraphs of $G$ ($\vert V_i\vert<\vert V\vert,i\in\{1,2\}$).

\begin{fact}
$G$ has a realization if and only if $G_1$ and $G_2$ have a realization.
\end{fact}
\begin{proof}
Obviously, if $G$ has a realization also $G_1$ and $G_2$ have it.
Therefore we focus on sufficiency. 

The following cases arise:

\begin{itemize}
    \item suppose that $K_n ^1$ and $K_m ^2$ are positive and that the vertices of $K_n ^1$ have the labels $\{1,2,a_k\},1\le k\le n$ with $v_1=\{1,2,a_1\}$, while the vertices of $K_m ^2$ have the labels $\{2,3,b_k\},1\le k\le m$  with $v_2=\{2,3,b_1\}$.  Taking $a_1=b_1$ and $a_i \neq b_j$ with $1<i\leq n$ and $1<j\leq m$, we obtain a realization for $G$;

    \item suppose $K_n ^1$ is positive and $K_m ^2$ is negative. Suppose also that $K_n ^1$ vertices have labels $\{1,2,a_k\},1\le k\le n$ with $v_1=\{1,2,a_1\}$. 
    
    From Fact \ref{K4+v} we have $m =3$. W.l.o.g., the vertices of $K_m ^2$ have the labels $\{2,3,4\},\{2,3,5\},\{3,4,5\}$ with $v_2=\{2,3,4\}$. Taking $a_1=4$  we obtain a realization for $G$; 
    
    \item suppose $K_n ^1$ is positive and $m =2$. Suppose also that the vertices of $K_n ^1$ have labels $\{1,2,a_k\},1\le k\le n$ with $v_1=\{1,2,3\}$. W.l.o.g., for the realization of $G_1$ we have $v_2=\{1,3,4\}$, then up to a renaming of the labels for the realization of $G'_2$  we have a realization for $G$; 
    
    \item suppose $K_n ^1$ and $K_m ^2$ are negative. From Fact \ref{K4+v} we have $n= m =3$. For the realization of $G_1$ let $v_1=\{1,2,3\},v_2=\{2,3,5\}$ and let $\{1,2,4\},\{1,3,4\}$ be the labels of the two other vertices of $K_n ^1$. Then up to a renaming of the labels for the realization of $G'_2$ with $\{3,5,6\},\{2,5,6\}$ the labels of the two other vertices of $K_m ^2$, we have a realization for $G$; 
    \item Suppose $n = m =2$.  For the realization of $G_1$ let $v_1=\{1,2,3\},v_2=\{1,3,5\}$ and let $\{1,2,4\}$ be the label of the other vertex of $K_n ^1$. Up to a renaming of the labels for the realization of $G'_2$ we have a realization for $G$.
    
\end{itemize}

\end{proof}

\vspace{1mm}
{\it G has no edge-cut:}
\vspace{1mm}

If $G$ has no edge cut, then the following Fact is readily obtained.

\begin{fact}
$G$ has no small cliques.
\end{fact}

 \begin{proof}
For contradiction we assume that $K=G[\{v_iv_j\}]$ is a clique with $v_i\in K_i,v_j\in K_j$, where $K_i\ne K,K_j\ne K,$ are two distinct cliques. Since $v_iv_j$ is not a cut-edge there exists a path $v_i-v_k-\cdots-v_j$ where  $v_k\ne v_j$. Assume that it is one of the shortest paths. Then $v_i-v_k-\cdots-v_j-v_i$ is an induced cycle of length greater than three, a contradiction.
\end{proof}

Therefore, we suppose without loss of generality that $G$ contains only medium or big cliques.

Using the algorithm given in \cite{maxcliq} we obtain $T=(C,S)$, a maximal clique tree of $G$ in time $O(n)$ where a vertex $c\in C$ corresponds to a big or a medium clique of $G$. We show how to build a labelling of it. 

First, we need the following fact.

\begin{fact}
    Let $c\in C$ with $\vert c\vert=3$. Then, $c$ has at most $3$ neighbours in $T$.
\end{fact}

\begin{proof}
Suppose that $c$ has more than $3$ neighbours. Then, considering our previous results, only two cases are possible: at least two cliques weakly intersect $c$ in the same node $v$ or one strongly intersects $c$ in $\{v,w\}$ and the other weakly intersects $c$, without loss of generality, in $\{v\}$.
We consider these cases separately:

\begin{itemize}

    \item  assume that $c$ has two neighbors $c_1,c_2$ that weakly intersect it in $v\in c$. Let $w\in c,w\ne v$. Since the cliques are distinct and of size greater than one, there must exist $v_i\in c_i,v_i\ne v,i\in \{1,2\}$. Remember that, if $c_1,c_2$ are neighbours of $c$, it means that $\{v\}$ is a separator in $G$. Therefore, $wv_1,wv_2,v_1v_2\not\in E$, but $G[\{v,w,v_1,v_2\}]=K_{1,3}$, a contradiction;
    \item we assume that $c$ has two neighbors $c_1,c_2$ such that $c_1 \cap c = \{v,v'\}$ and $c_2 \cap c = \{ v \}$. Let $w\in c,w\ne v,v'$. 
    Since $\{v,v'\}$ and $\{v\}$ are two separators $wv_1,wv_2,v_1v_2\not\in E$, but $G[\{v,w,v_1,v_2\}]=K_{1,3}$, a contradiction. 
\end{itemize}

Thus $c$ has at most three neighbours. 
\end{proof}
 
Note that for the case where $c$ has three neighbours, either $c$ is the central triangle of sun or it weakly intersects three cliques $c_1,c_2,c_3$ in each of its vertices.

Returning to the main problem, we will associate a label, $\lambda(c)\in\{+,-\}$ to each vertex $c$ of $T$. 
The idea is that the label associated with a clique represents its negativity or positivity. 

In the first stage, partial labelling is obtained with the Preprocessing Labelling algorithm. In this step only the {\it local} forcings based on big cliques, Facts \ref{K4+v} and \ref{sun} are taken into account. The last instruction  of the algorithm consists of checking the Fact \ref{butter}.

\begin{algorithm}\label{preproc}
\caption{Preprocessing Labelling}
\begin{algorithmic}
\FORALL{$c\in T$} 
 \IF{ $\vert c\vert\ge 5$}\STATE $\lambda(c)=+$ \COMMENT{a big clique is positive} \ENDIF 
 \IF{$\vert c\vert=4$ and there exists $c'$ that weakly intersects $c$ }\STATE $\lambda(c)=+$ \COMMENT{by Fact \ref{K4+v}}\ENDIF 
\IF{$\vert c\vert=3$ and there are $c_1,c_2,c_3$ such that $\vert c_1\vert=\vert c_2\vert=\vert c_3\vert=3$ and $c_1,c_2,c_3$ strongly intersects $c$ in different arcs}
\STATE $\lambda(c)=+$\COMMENT{by Fact \ref{sun} noticing that $c,c_1,c_2,c_3$ induce a sun}
\FOR{i=1\TO 3 }\STATE {$\lambda(c_i)=-$} \ENDFOR
\ENDIF 
 \ENDFOR
 \FORALL{$cc'\in S$}
  \IF{$\lambda(c)=\lambda(c')=-$}\RETURN {$G\not\in L_3^2$} \ENDIF \COMMENT{by Fact \ref{butter}}
 \ENDFOR
\end{algorithmic}
\end{algorithm}

Then, in the second stage, the algorithm Labelling propagates, from the bottom to the top, the labelling in  the neighbour of the cliques already labelled. The last instruction  of the algorithm consists of checking the Fact \ref{butter}.
Note that in the algorithm we are not considering the case in which a clique does not have a label, although it's possible that happen.

\begin{algorithm}
\caption{Labelling}
\begin{algorithmic}
\STATE choose $c_r\in C$ as the root of the directed tree $T_r$
\STATE  \COMMENT{for each $c\ne c_r$ the unique path $c_r-\cdots -c$ is a directed path starting from $c_r$}
\FORALL{$c\in T_r$ \AND $\lambda(c)=\emptyset$, from the leaves to $c_r$} 
 \STATE Let $c'$ be the predecessor of $c$ in $T_r$ 
 \IF{$c$ is a leaf}

 \IF{$c$ and $c'$ strongly intersect} 
 \IF {$\lambda(c')=+$}
 \STATE $\lambda(c)=-$\ENDIF 
  
  \IF {$\lambda(c')=-$}\STATE $\lambda(c)=+$\ENDIF
  \ENDIF 
   \IF{$c$ and $c'$ weakly intersect} \IF {$\lambda(c')=-$}\STATE $\lambda(c)=+$\COMMENT{in this case $|c'| = 3$ must hold}
 \ENDIF 
   \IF {$\lambda(c')=+$}\STATE $\lambda(c)=-$\ENDIF
  \ENDIF 
  \ENDIF 
  
  \IF{$c$ is not a leaf}
   \IF{($c$ has two successors $c_1,c_2$ that are strongly connected with $c$ \AND $\lambda(c_1)\ne\lambda(c_2)$) \OR ($c$ has a successor $c_1$ that is strongly connected with $c$, $\lambda(c_1)=+$ \AND $c$ has a successor $c_2$ that is weakly connected with $c$, $\lambda(c_2)=-$ )}\RETURN {$G\not\in L_3^2$} \ENDIF 
 \IF{$c$ has a successor $c_1$ strongly connected with $c$} 
   \IF {$\lambda(c_1)=+$}\STATE $\lambda(c)=-$\ENDIF
    \IF {$\lambda(c_1)=-$}\STATE $\lambda(c)=+$\ENDIF
  \ENDIF
   \IF{$c$ has a successor $c_1$ weakly connected with $c$ \AND
    $\lambda(c_1)=-$}\STATE $\lambda(c)=+$\ENDIF
    
  \IF{$\lambda(c)=\lambda(c')=-$}\RETURN {$G\not\in L_3^2$} \ENDIF 
  \ENDIF 
  \ENDFOR
\end{algorithmic}
\end{algorithm}

We consider these cases in the last stage. In fact, when some vertices are not labelled, we can fix their labels to either $+$ or $-$ in such a way that the labels alternate for the cliques $c,c'$ that strongly intersect.
 
Note that the labelling of $T_r$ terminates in time $O(n)$.

Then we find a $3$-uniform hypergraph $H$ such that $G=L_3^2(H)$.

\vspace{5mm}
{\it Construction of the $3$-uniform hypergraph}

\begin{fact}
Consider the labelled tree $T$ of a claw-free triangulated hypergraph. 
Then there exists $H$ such that $G=L_3^2(H)$.
\end{fact}

\begin{proof}
Given $T$ with the labelling of its vertices, we construct a labelling of the vertices of $G$ from the root $c_r$ to the leaves of $T_r$. We assume w.l.o.g. that $\lambda(c_r)=+$.
We label the vertices of $c_r$ as follows: since $\lambda(c_r)=+$, $c_r$ it is labelled positively with the labels $(1,2,3)$, $(1,2,4)$,$\ldots$,$(1,2,k_r)$. 
Let $c\in C,c\ne c_r$. We assume that the vertices of its predecessor $c'$ in $T_r$ are labelled.

Since $T$ is a tree, up to a permutation of the labels, we assume that the labels of $c'$ are taken into $\{1,2,\ldots,k\}$.
Consider the following two cases:

\begin{itemize}
    \item Suppose that $c'$ is labelled positively with $(1,2,3),(1,2,4),\ldots,(1,2,k)$ . 
    If $c$ strongly intersect $c'$ then 
    $\lambda(c)=-$. We set $(1,2,3),(1,2,4)$ the labels of $c\cap c'$. If $c$ contains three vertices, it is labelled as $(1,3,4)$.  If it contains four vertices, the last node is labelled as $(2,3,4)$. 
    
    If $c$ weakly intersect $c'$  from Fact \ref{K4+v} we have $\vert c\vert=3$. Let $(1,2,3)$ be the label of $c\cap c'$. If $\lambda(c)=-$ then the other labels of $c$ are $(1,3,a),(2,3,a)$, with $a$ different from any value in $c$ and $c'$. If $\lambda(c)=+$ then  the other labels of $c$ are $(1,3,a),(1,3,b)$, with $a,b$ different from any value in $c$ and $c'$. In any case, we obtain a valid label for the two cliques;

    \item Suppose that $c'$ is labelled negatively. 
    % with $(1,2,3),(1,2,4),(1,3,4),(2,3,4)$ if $\vert c'\vert=4$, or $(1,2,3),(1,2,4),(1,3,4)$ if $\vert c'\vert=3$.
    % Then, it must holds $\lambda(c) = (+)$.
    Let's consider the case in which $\vert c'\vert= 4$ and their vertices are $(1,2,3),(1,2,4),(1,3,4),(2,3,4)$: from Fact \ref{K4+v} $c$ strongly intersect $c'$ and $\lambda(c) = +$ must holds. Let $(1,2,3),(1,2,4)$ be the labels of $c\cap c'$. The other labels of $c$ are $(1,2,5),\ldots,(1,2,k)$. 
    
    Consider now the case in which $\vert c'\vert=3$ and their vertices are $(1,2,3),(1,2,4),(1,3,4)$. If $c$ strongly intersect $c'$, let $(1,2,3),(1,2,4)$ be the labels of $c\cap c'$. The other labels of $c$ are $(1,2,5),\ldots,(1,2,k)$. 
    
    On the other hand, if $c$ weakly intersects $c'$, let $(1,2,3)$ be the label of $c\cap c'$. The other labels of $c$ are $(2,3,5),\ldots,(2,3,k)$. 
    
    fIn case $c'$ has a second successor $c''$, such that the label of $c\cap c''$ is $(1,2,4)$, then the other labels of $c''$ are $(2,4,k+1),\ldots,(2,4,k+l)$. Recall that  $c'\ne c_r$ so $c'$ has at most two successors in $T_r$.
\end{itemize}
\end{proof}

To conclude, we estimate the complexity of the whole procedure. When $G$ has no cut edge the procedure takes a time $O(n)$. The algorithm in \cite{HT} gives the edge cuts and the biconnected components in time $O(m)$ when $G$ is connected. Hence the complexity of the algorithm is $O(m)$.

The proof is completed. 
\end{proof}

As an example, consider the graph depicted $G$ in Figure~\ref{fig:clique_tree} and its associated tree cliques.
Using the rules listed in the proof, we obtain the labelling shown in Figure~\ref{fig:labelled_tree}. From that label, it's easy to find an actual label of the vertices and conclude that $G \in L_3^2(H)$.

\begin{figure}
    \centering
    \includegraphics[width = \linewidth]{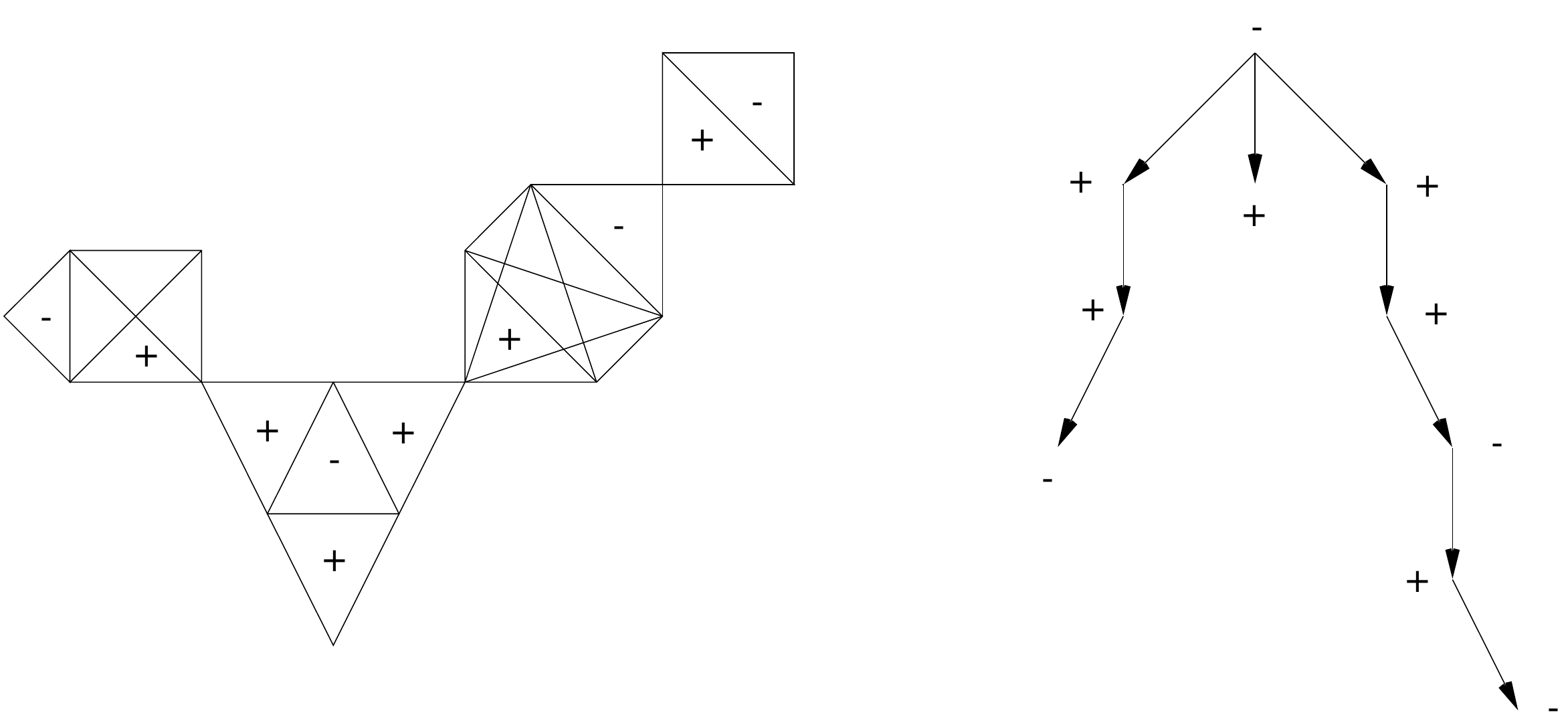}
    \caption{Labelling of the graph $G$ and its associated tree.}
    \label{fig:labelled_tree}
\end{figure}

\section{Further properties of $L_k^l$ class}\label{prop}

In this section we prove some further properties of graph belonging to certain $L_k^l$ classes. After giving some properties on the class $L_k^1$, we move to $L_3^2$ providing a complexity result for the hamiltonian cycles detection problem. In particular, Di Marco et al. proved in \cite{rec2inter} that hamiltonian cycles in $L_3^2$ are useful to assure the existence of a null label in the relative hypergraph.
We start giving some properties for a graph $G\in L_k^l$. 

\begin{fact}\label{pcliques}
If $G\in L_k^l$ then $G$ is $K_{1,p+1}$-free, $p=C_{k}^{l}$.
\end{fact}

\begin{proof}
Let $G\in L_k^l$ and $H$ a preimage of $G$.  Suppose that $G$ contains $K_{1,p+1}$ with a central vertex $v$ and $N(v) = \{e_1\ldots,e_{p+1}\}$. Then two edges $e_i,e_j,1\le i<j\le p+1$  exit such that $\vert e_i\cap e_j\vert\ge p$, so $e_ie_j \in E$, a contradiction.
\end{proof}

\subsection{NP-complete problems in the class $L^1_k$}
Given a linear $k$-uniform hypergraph $H=(V,E)$ one can build a linear $(k+1)$-uniform hypergraph $H'=(V',E')$ as follows. For each hyperedge $e_i\in E,1\le i\le m,$ we create the hyperedge $e'_i\in E',e'_i=e_i\cup\{a_i\}$. Then we take $V'=V\cup_{1\le i\le m}\{a_i\}$.

From this construction and the NP-complete results we give below we obtain the following properties.

In \cite{rec1inter} it is proved that recognizing whether a graph $G\in L_3^1$ is $NP$-complete. Hence we obtain.

\begin{prop}\label{NPRec}
For any fixed $k\ge 3$, deciding wether a graph is n $L_k^1$ is  NP-complete.
\end{prop}

 Other significant NP-complete results for the line graphs are the following.
\begin{prop}\label{NPLin3}
The problems Hamiltonian cycle \cite{HCLine},  $3$-coloring \cite{3CLine},  Minimum domination \cite{DomLine} are NP-complete in $L_2^1$.
\end{prop}

It follows from Proposition \ref{NPLin3}  the next complexity results.
\begin{prop}\label{NPLink}
For any fixed $k\ge 2$, the problems Hamiltonian cycle, $3$-coloring, Minimum domination are NP-complete in $L_k^1$.
\end{prop}

\subsection{Hamiltonian cycle detection in $L_3^2$}\label{hamilton}
In \cite{nullLabel} authors study the null label problem and prove a sufficient condition for a $3-$hypergraph to be null. In particular, the result uses Hamiltonian graphs in $L_3^2$. Here, we show that deciding if $G\in L_3^2$ is Hamiltonian is $NP$-complete, limiting the possible application of the result.

\begin{theorem}
The Hamiltonian cycle problem is $NP$-complete in $L_3^2$ even for graphs $G=L_3^2(H)$ where $m(H)=3$.
\end{theorem}

\begin{proof}
We give a polynomial transformation from the Hamiltonian cycle problem in cubic graphs which is $NP$-complete \cite{GJ}. From a cubic graph $G'=(V',E')$, we define  $G\in L_3^2$ as follows: to each vertex $v\in V'$ with neighbours $u$, $w$, and $t$ corresponds $K_v$, i.e., the complete graph with the three vertices $(v,{\overline v},  u),(v,{\overline v},  w),(v,{\overline v},  t)$. For each edge $uv\in E'$, we add the edge $(v,{\overline v},  u)(u,{\overline u},  v)$. It is straightforward to verify that  $G\in L_3^2$, and that $H$, the hypergraph such that $G=L_3^2(H)$, satisfies $m(H)=3$. Moreover, $G'$ has a hamiltonian cycle if and only if $G$ has one.
\end{proof}

\begin{rmk}
In \cite{GJT} is proved that the Hamiltonian cycle problem  remains $NP$-complete for cubic planar graphs. Since $K_3$, the subgraph replacing each vertex in our reduction is planar, it is  straightforward, using the same transformation, that the Hamiltonian cycle problem is $NP$-complete in $L_3^2$ even for planar graphs.
\end{rmk}

\subsection{Recognition problem for trees}\label{tree}

We are interested in the recognition problem for trees in $L_3^2$. 

\begin{prop}
Let $T$ be  a tree. $T\in L_3^2$ if and only if $\Delta(T)\le 3$. 
\end{prop}
%\andrea{is there a problem in the property statement.}
\begin{proof}
Let $T$ be a tree. If $\Delta(T)\ge 4$ then $T$ contains $K_{1,4}$ as an induced subgraph and so $T\not\in L_3^2$. Now $\Delta(T)\le 3$. We use induction on $n$, the number of vertices of $T$. The cases $n=1$ and $n=2$ are trivial. Let $v$ be a leaf of $T$. By our induction hypothesis, $T-v$ has a $\lambda_3^2$-labelling.
Let $w$ be the neighbour of $v$ in $T$. In $T-v$, $w$ has degree at most two. Without loss of generality, let $w=\{1,2,3\}$ and $w'=\{1,2,4\}$ be the labelling of a neighbor $w'$ of $w$ in $T-e$. When $w'$ is the unique neighbor of $w$ in $T_e$ then $v=\{2,3,5\}$. Else $w''$ is the second neighbour of $w$ in $T_e$. Let $w''=\{1,3,5\}$. Then $v=\{2,3,6\}$. 
\end{proof}

% \section{Conclusion}\label{concl}

\end{document}